\documentclass[11pt]{amsart}

\usepackage{float}
\usepackage{graphicx} 
\usepackage{xcolor}
\usepackage{amsmath, amssymb, amsthm, amsfonts, enumerate, dsfont, mathtools, hyperref, comment}
\usepackage[resetlabels,labeled]{multibib}
\newcommand{\ie}{i.e.\ }
\newcommand{\eg}{e.g.\ }
\numberwithin{equation}{section}

\newcommand{\MLS}[0]{{\rm MLS}}

\renewcommand{\phi}{\varphi}

\newcommand\R{\mathds{R}}
\newcommand\K{\mathbb{K}}

\theoremstyle{plain}
\newtheorem{thm}{Theorem}[section]
\newtheorem{cor}[thm]{Corollary}
\newtheorem{lem}[thm]{Lemma}

\newtheorem{conj}[thm]{Conjecture}
\newtheorem{question}[thm]{Question}

\usepackage[parfill]{parskip}
\newtheorem{TOD}{TODO}

\theoremstyle{remark}
\newtheorem{rem}[thm]{Remark}

\title[Magnetic MLS rigidity for surfaces]{Marked length spectrum rigidity \\ for Anosov magnetic surfaces}
\author[V. Assenza, J. de Simoi, J. Marshall Reber, I. Terek]{Valerio Assenza \and Jacopo de Simoi, \\ James Marshall Reber \and Ivo Terek}

\address{Instituto de Matem\'{a}tica Pura e Aplicada, Rio de Janeiro, RJ, 22460-320, Brazil}
\email{valerio.assenza@impa.br}

\address{Department of Mathematics, University of Toronto, Toronto, ON, Canada}
\email{jacopods@math.utoronto.ca}

\address{Department of Mathematics, The Ohio State University, Columbus, OH, 43210, USA}
\email{marshallreber.1@osu.edu}

\address{Department of Mathematics and Statistics, Williams College, Williamstown, MA, 01267, USA}
\email{it3@williams.edu}

\begin{document}
	
	\begin{abstract}
		We show that if $M$ is a closed, connected, oriented surface, and two Anosov magnetic systems on $M$ are conjugate by a volume-preserving conjugacy isotopic to the identity, with their magnetic forms in the same cohomology class, then the metrics are isometric. This extends the recent result by Guillarmou, Lefeuvre, and Paternain to the magnetic setting.
	\end{abstract}
	\maketitle

	\section{Introduction} \label{sec:intro}
	Let $M$ be a closed,
	connected, oriented manifold, and let $g$ be a Riemannian metric
	on $M$. Denoting the corresponding geodesic flow on its unit
	tangent bundle by $\phi_g^t : S_gM \rightarrow S_gM$ and denoting
	the infinitesimal generator of the flow by $X^g$, we say that the
	metric $g$ is \emph{Anosov} if $\phi_{g}^{t}$ is an \emph{Anosov flow},
	i.e., there exists a $\phi_g^t$-invariant splitting
	$TS_gM = E^+ \oplus \R X^g \oplus E^-$ and constants $c, d > 0$ so
	that for all $t \geq 0$ and $\xi \in E^\pm(v)$, we have
	\begin{equation}\label{Anosovcondition}
		\|d_v\phi_g^{\mp t}(\xi)\| \leq d e^{-ct} \|\xi\|,
	\end{equation}
	where $\|\cdot\|$ denotes the norm coming from the Sasaki metric
	on $S_gM$.
	
	Anosov metrics can be seen as a generalization of negatively
	curved metrics \cite{anosov} and, as such, they share many common
	qualities. For example, observe that if $g$ is an Anosov
	metric, then for every non-trivial free homotopy class
	$\nu \in \pi_1(M)$, there is a unique closed $g$-geodesic
	$\gamma_\nu \in \nu$ \cite{klingenberg}. The \emph{marked length
		spectrum} is defined to be the function
	${\MLS(g) : \pi_1(M) \rightarrow [0,\infty)}$ which assigns to a
	non-trivial free homotopy class $\nu$ the length of its geodesic
	representative $\gamma_\nu$, and $0$ to the trivial free homotopy
	class. It is conjectured that $\MLS(g)$ uniquely characterizes the
	Anosov metric $g$. More precisely:
	\begin{conj}[\cite{guillarmou2023marked}]\label{MLSconjecture}
		Let $g_1$ and $g_2$ be two Anosov metrics on $M$. If ${\rm MLS}(g_1) = {\rm MLS}(g_2)$, then $g_1$ and $g_2$ are isometric by an isometry which is isotopic to the identity.
	\end{conj}
	In the case where $g_1$ and $g_2$ are both negatively curved
	metrics, this is the well-known \emph{Burns-Katok
		conjecture} \cite{burns1985manifolds}. Using Livshits' theorem
	\cite{livshits1972cohomology} along with an argument by Gromov \cite{gromov2000three}, one can show that if
	$\MLS(g_1) = \MLS(g_2)$, then there is a $C^0$-conjugacy between
	the corresponding geodesic flows which is isotopic to the
	identity. With this in mind, the conjecture broadly states that
	dynamical invariants give rise to geometric rigidity. It is known that
	Conjecture ~\ref{MLSconjecture} holds for non-positively curved
	metrics on surfaces \cite{croke1990rigidity, croke1992marked,
		otal1990spectre}, and for negatively curved metrics in
	dimensions three and higher when one of the metrics is locally
	symmetric \cite{BCG, hamenstdt1999cocycles}. For Anosov metrics,
	the conjecture has been recently proved for surfaces in
	\cite{guillarmou2023marked}; a local version (i.e.\ when the
	metrics are sufficiently close in an appropriate topology) holds
	in arbitrary dimension, provided the metrics have non-positive sectional curvature \cite{guillarmou2019marked}. The conjecture
	is still open in general.
	
	In this paper, we bring the above discussion to the context of
	magnetic dynamics. Namely, we study marked length spectrum
	rigidity for Riemannian metrics on surfaces in presence of a
	magnetic field.  Magnetic dynamics was formalized in
	\cite{anosov1967some, arnold2009some}, and has been investigated
	in many works in the last decades. For a general overview, we
	refer the reader to \cite{IVJ}, and the references
	therein.
	
	From now on, let $M$ be a closed, connected, orientable surface. A \textit{magnetic system} is a pair $(g,b)$, with $g$ a Riemannian metric on $M$ and $b$ a smooth function (referred to as the \textit{magnetic intensity}) on $M$. 
	Denote by $\Omega_g$ the area form induced by $g$ on $M$, and let $i$ be the complex structure given by the orientation of $\Omega_g$.
	More precisely, given a vector $v \in TM$, $iv$ is the rotation of $v$ by the angle $\pi/2$ according to the orientation.
	A curve $\gamma : \R \to M$ is a \textit{$(g,b)$-geodesic} if it satisfies the differential equation
	\begin{equation} \label{eqn:magnetic_defn}
		\frac{\mathrm{D} \dot{\gamma}}{\mathrm{d}t}(t) =  b(\gamma(t)) i \dot{\gamma}(t),
	\end{equation}
	where ${\rm D}/{\rm d}t$ denotes the covariant derivative along $\gamma$ induced by the Levi-Civita connection of $g$. A $(g,b)$-geodesic describes the motion of a charged particle on $(M,g)$ under the influence of a magnetic force described by $b$.
	
	Observe that
	if $b =0$, then \eqref{eqn:magnetic_defn} reduces to the standard equation for \hbox{$g$-geodesics} on $M$. Also observe that while $(g,b)$-geodesics have constant speed, they are not homogeneous if $b \neq 0$. In particular, rescaling the initial speed vector or inverting the direction of the initial velocity vector may drastically change the associated trajectory. Despite this, up to a time reparameterization, one can recover the dynamics of unit speed $(g,b)$-geodesics from the dynamics of $(g,b)$-geodesics with speed $s > 0$ by rescaling the metric $g$ by the factor $1/s^2$ and the magnetic intensity by the factor $1/s$. Without loss of generality, we restrict our perspective to unit speed solutions of \eqref{eqn:magnetic_defn}.

	The \emph{magnetic flow} associated to the pair $(g,b)$ is the flow $\varphi_{g,b}^t : S_gM \rightarrow S_gM$ given by $\phi_{g,b}^t(\gamma(0),\dot{\gamma}(0)) \coloneqq (\gamma(t), \dot{\gamma}(t))$.
	The orbits of $\phi_{g,b}^t$ have a geometric interpretation in terms of the geodesic curvature. The \emph{$g$-geodesic curvature} associated to a curve $\gamma : \R \rightarrow M$ is defined by
	\[\mathrm{k}^g(t) \coloneqq g \left( \frac{\mathrm{D} \dot{\gamma}}{\mathrm{d}t}(t),i\dot{\gamma}(t)\right).\]
	It is immediate from \eqref{eqn:magnetic_defn} that a unit speed curve $\gamma$ is a $(g,b)$-geodesic if and only if the $g$-geodesic curvature of $\gamma$ satisfies $\mathrm{k}^g(t)=b(\gamma(t)) $.

	The magnetic system $(g,b)$ is called an \textit{Anosov magnetic system} if the flow $\phi_{g,b}^t$ is Anosov. Similar to the geodesic scenario, the existence of an Anosov magnetic system on $M$ implies some restrictions on both the magnetic intensity and the surface. By \cite[Corollary C]{yumin}, the genus of $M$ must be at least two and, as pointed out in \cite[Theorem~B]{burns2002anosov}, the average of the magnetic intensity cannot be larger than the so-called \emph{Ma\~{n}\'{e} critical value}.
	
	An interesting problem is determining how to construct Anosov magnetic systems. By the structural stability of Anosov flows, observe that if $g$ is an Anosov metric, then $(g,b)$ is Anosov as long as $b$ is uniformly small. On the other hand, not every Anosov magnetic system comes from an Anosov metric: in \cite[Section~7]{burns2002anosov}, it was shown that one can start with a non-Anosov metric on a higher genus surface and place a magnetic intensity $b$ on $M$ so that $(g,b)$ is Anosov. This shows that the set of metrics for which there is a magnetic intensity $b$ so that $(g,b)$ is Anosov is larger than the set of Anosov metrics.

	
	As in the Riemannian setting, an important subset of the set of Anosov magnetic systems consists of those which are negatively curved in a magnetic sense.
	The \emph{magnetic curvature} associated to the system $(g,b)$ is given by
	\begin{equation} \label{eqn:curvature_defn} \K^{g,b} : S_gM \rightarrow \R \quad \K^{g,b}(x,v) \coloneqq K^g(x) - \mathrm{d}_xb(iv) + b^2(x), \end{equation}
	where $K^g$ is the Gaussian curvature of $(M,g)$. It was shown in \cite{wojtkowski2000magnetic} that if $\K^{g,b} < 0$, then the magnetic flow is Anosov.\footnote{See also \cite{gouda, grognet1999flots} for earlier works in this direction, and \cite[Appendix A]{IVJ} for the proof in higher dimensions.} Under additional assumptions, this result has been extended to non-positively curved magnetic flows in the recent paper \cite{yumin} of the third author in collaboration with Yumin Shen.
	%
	
	Given an Anosov magnetic system $(g,b)$, it is still true that for each non-trivial free homotopy class, there is a unique closed $(g,b)$-geodesic with unit speed \cite{contreras2000palais}. With this in mind, we define the \textit{magnetic marked length spectrum} as the function $\mathrm{MLS}(g,b):\pi_1(M) \to [0,+\infty)$ which associates to a non-trivial free homotopy class $\nu$ the length of the unique closed unit speed $(g,b)$-geodesic $\gamma_{\nu}$ in $\nu$, and zero to the trivial free homotopy class.
	\begin{rem}
		Observe that in the Riemannian case, the traces of
		$\gamma_{\nu}$ and $\gamma_{-\nu}$ always coincide.  This is not
		necessarily the case for magnetic geodesics. In particular, for
		some $\nu$, it may be that
		$\mathrm{MLS}(g,b)(\nu) \neq \mathrm{MLS}(g,b)(-\nu)$, although
		there are non-trivial examples where there is
		equality for every choice of $\nu$ (\eg Lemma~\ref{lem:area}).
	\end{rem}
	
	Our main result shows that the magnetic marked length spectrum, the area, and the cohomology class of $[b \Omega_g]$ determine the Riemannian metric and the magnetic intensity, extending the result of the third author in \cite{MR} from the deformative setting to the global setting.
	
	\begin{thm} \label{thm:main1}
		Let $M$ be a closed, connected, oriented surface. Suppose $(g_1,b_1)$ and $(g_2,b_2)$ are Anosov magnetic systems on $M$ satisfying the following conditions:
		\[ {\rm MLS}(g_1,b_1) = {\rm MLS}(g_2,b_2), \ {\rm Area}(g_1) = {\rm Area}(g_2), \ \mbox{and} \ [b_1 \Omega_{g_1}] = [b_2 \Omega_{g_2}].\]
		Then there exists a diffeomorphism $f : M \rightarrow M$ which is isotopic to the identity and which satisfies $f^*(g_2) = g_1$ and $f^*(b_2) = b_1$.
	\end{thm}

	It follows from Livshits' theorem \cite{livshits1972cohomology} along with a theorem of Ghys \cite{ghys1984flots} that equality of the marked length spectrum implies that there is a {$C^0$-conjugacy} between the corresponding magnetic flows which is isotopic to the identity. In the case where the magnetic intensities are trivial, one can use the work of Feldman-Ornstein \cite{feldman} to deduce that the conjugacy is smooth. The recent rigidity result by Gogolev-Rodriguez Hertz in \cite[Theorem 1.1]{gogolev2022smooth} extends this work to a larger class of Anosov flows on three manifolds, and, in particular, shows that  
	equality of the magnetic marked length spectrum
	implies that the flows $\phi_{g_1,b_1}^t$ and $\phi_{g_2,b_2}^t$ are $C^\infty$-conjugate by some $h : S_{g_1}M \rightarrow S_{g_2}M$ which is isotopic to the identity (cf.\ \cite{wilkinson}).
	
	Let $\alpha^g$ be the contact form associated to the geodesic flow $\varphi_g^t$. Recall that the \emph{Liouville volume form} on $S_gM$ associated to the metric $g$ is the top-degree form $\mu^g \coloneqq - \alpha^g \wedge {\mathrm d}\alpha^g$. As we will see in Lemma ~\ref{lem:reduction}, one can use the area assumption to show that the conjugacy described above must pullback the Liouville volume form for $g_2$ to the Liouville volume form of $g_1$, and thus we can reduce Theorem ~\ref{thm:main1} to the following.

	\begin{thm} \label{thm:main2}
		Let $M$ be a closed, connected, oriented surface. Suppose $(g_1,b_1)$ and $(g_2,b_2)$ are Anosov magnetic systems on $M$ such that the corresponding magnetic flows are smoothly conjugate by a conjugacy $h$ which is isotopic to the identity, the conjugacy $h$ preserves the corresponding Liouville volume forms, and $[b_1 \Omega_{g_1}] = [b_2 \Omega_{g_2}]$.
		Then there is a diffeomorphism $f : M \rightarrow M$ which is isotopic to the identity and which satisfies $f^*(g_2) = g_1$ and $f^*(b_2) = b_1$.
	\end{thm}

	The proof of Theorem ~\ref{thm:main2} is as follows. First, we use a recent rigidity result by Echevarr\'{i}a Cuesta in \cite[Theorem 1.1]{javier} to reduce the problem to the setting where $g_1$ and $g_2$ are conformally related. Once in this setting, we adapt the conformal argument of Katok in \cite{katok} in the magnetic setting. In the case where $[b_1 \Omega_{g_1}] = [b_2 \Omega_{g_2}] = 0$, the argument follows by using the fact that magnetic geodesics are minimizers for a functional which we call the \emph{magnetic length}. In the non-exact setting, we use the fact that the flow is homologically full to work on homologically trivial orbits, and then take advantage of the fact that $b_1 \Omega_{g_1}$ and $b_2 \Omega_{g_2}$ are exact when pulled back to the unit tangent bundle.
	
	We note that the assumptions on the area and the cohomology class in Theorem ~\ref{thm:main1} are necessary via two examples in Section ~\ref{section:examples}, which can also be found in \cite{paternain2006magnetic}. For the area assumption, we show in Lemma ~\ref{lem:area} that for a hyperbolic metric $g$ and a constant $|b| < 1$, we have $\MLS(g,b) = (1-b^2)^{-1/2} \MLS(g)$. By considering the homothetic metric $g_b \coloneqq (1-b^2)^{-1}$, we see that $\MLS(g, b) = \MLS(g_b)$ while $\text{Area}(g) \neq \text{Area}(g_b)$. It follows that the marked length spectrum assumption alone is not enough to guarantee that the underlying metrics are isometric.
	
	We show in Lemma ~\ref{lem:homology_class} that $[b_1 \Omega_{g_1}] = \pm [b_2 \Omega_{g_2}]$ (see also \cite[Lemma 4.1]{paternain2006magnetic} and \cite[Proposition 4.5]{javier}). Lemma ~\ref{lem:intensity} shows that this is actually realizable, in the sense that for a hyperbolic metric $g$ and a constant $|b| < 1$, the magnetic systems $(g,b)$ and $(g,-b)$ are smoothly conjugate by a conjugacy which is isotopic to the identity. In particular, this shows that the marked length spectrum assumption alone is not sufficient to guarantee that the cohomology class of $b_1 \Omega_{g_1}$ and $b_2 \Omega_{g_2}$ are the same.
	
	We now give some applications of Theorem ~\ref{thm:main1}. First, we observe that Theorem ~\ref{thm:main1} extends \cite[Theorem 1.1]{guillarmou2019marked} to the case of \emph{exact magnetic systems}, i.e., magnetic systems where $[b \Omega_g] = 0$. Using Lemma \ref{lem:homology_class}, it can be shown that if $(g_1, b_1)$ and $(g_2, b_2)$ are exact Anosov magnetic systems which are smoothly conjugate, then $\text{Area}(g_1) = \text{Area}(g_2)$. In fact, this also extends to $(g_1, b_1 + \delta)$ and $(g_2, b_2 + \delta)$, provided that $\delta \in \R$ is sufficiently small so that the marked length spectrum still makes sense. With this in mind, the following is immediate.
	
	\begin{cor}\label{cor:main}
		Let $M$ be a closed, connected, orientable surface, and let $(g_1,b_1)$ and $(g_2,b_2)$ be exact Anosov magnetic systems on $M$. If $\delta \in \R$ is sufficiently small so that $(g_1, b_1 + \delta)$ and $(g_2, b_2 + \delta)$ are still Anosov and ${\rm MLS}(g_1, b_1 + \delta) = {\rm MLS}(g_2, b_2 + \delta)$, then there exists a diffeomorphism $f : M \rightarrow M$ which is isotopic to the identity and which satisfies $f^*(g_2) = g_1$ and $b_2 \circ f = b_1$.
	\end{cor}
	
	Notice that there is no assumption on the metric in Corollary ~\ref{cor:main} -- as long as the exact magnetic systems $(g_i, b_i + \delta)$ share the same marked length spectrum, then the underlying metrics are isometric. As previously mentioned, notice that $\MLS(g,b)$ encodes the length of the unique closed curve $\gamma$ in each non-trivial free homotopy class of $M$ with prescribed $g$-geodesic curvature equal to $b \circ \gamma$. The above result can then be interpreted as saying that the unit speed curves with $g$-geodesic curvature equal to $b$ completely determine the underlying metric, provided the set of these curves is ``dynamically rich.'' In Section ~\ref{section:examples}, we modify the example in \cite[Section 7]{burns2002anosov} to construct an example of a metric $g$ which is not Anosov, but for which there exists an magnetic intensity $b$ so that $[b \Omega_g] = 0$ and the magnetic system $(g, b)$ is Anosov. Thus, Corollary ~\ref{cor:main} along with this example shows that marked length spectrum rigidity can be extended beyond the setting of Anosov metrics and non-positively curved metrics.

	Next, we see how the magnetic marked length spectra uniquely characterizes geodesic flows, provided the area is fixed.
	Using the arguments in \cite[Th\'{e}or\`{e}me 7.3]{grognet1999flots}, one can show that if $(g_1,b)$ and $g_2$ both satisfy $\K^{g_1,b} ,K^{g_2}<0$, $\text{MLS}(g_1,b) = \text{MLS}(g_2)$, and $\text{Area}(g_1) = \text{Area}(g_2)$, then $b \equiv 0$ and $g_1$ and $g_2$ are isometric. One can also deduce from Lemma \ref{lem:homology_class} that if $\MLS(g_1,b) = \MLS(g_2)$ and $\text{Area}(g_1) = \text{Area}(g_2)$, then $[b \Omega_g] = 0$.
	This shows that Corollary ~\ref{cor:main} extends \cite[Th\'{e}or\`{e}me 7.3]{grognet1999flots} from negatively curved magnetic systems to Anosov magnetic systems.
	
	\begin{cor} \label{cor:main2}
		Let $(g_1,b)$ be an Anosov magnetic system. If $g_2$ is an Anosov metric such that ${\rm MLS}(g_1,b) = {\rm MLS}(g_2)$ and ${\rm Area}(g_1) = {\rm Area}(g_2)$, then $b \equiv 0$ and there is an isometry between $g_1$ and $g_2$ which is isotopic to the identity.
	\end{cor}
	
	Finally, we introduce the magnetic marked length spectrum rigidity problem in higher dimensions; for more details on this setting, we direct this reader to \cite{assenza2023magnetic, IVJ}. Let $M$ be a closed, connected manifold of arbitrary dimension. A \emph{magnetic system} is a pair $(g,\sigma)$, where $g$ is a Riemannian metric and $\sigma$ is a closed $2$-form (referred to as the \emph{magnetic form}). Note that the magnetic form generalizes the magnetic intensity $b$ introduced for surfaces: if $\sigma$ is a closed $2$-form on a surface, then the magnetic intensity is precisely the unique smooth function $b$ on $M$ for which $\sigma = b \Omega_g$. A \emph{$(g,\sigma)$-geodesic} is a curve $\gamma : \R \to M$ which satisfies the equation
	\begin{equation}\label{sol}
		\frac{{\rm D} \dot{\gamma}}{{\rm d}t}(t) =  Y( \dot{\gamma}(t)),
	\end{equation}
	where $Y : TM \to TM$ is the \textit{Lorentz operator} defined by the relation $g( Y(v),w) = \sigma(v,w)$ for all $v,w\in TM.$
	As in the surface case, solutions of \eqref{sol} have constant speed, and the \emph{magnetic flow} $\phi_{g,\sigma}^t : S_gM \rightarrow S_gM$ associated to $(g,\sigma)$ is the flow on $S_gM$ induced by \eqref{sol}. A pair $(g,\sigma)$ is \emph{Anosov} if the flow $\phi_{g,\sigma}^t$ is Anosov, and as before the magnetic marked length spectrum naturally generalies to Anosov pairs $(g, \sigma)$. 
	With this in mind, a natural question is the following.
	
	\begin{question}
		Let $(g_1,\sigma_1)$ and $(g_2,\sigma_2)$ be two Anosov magnetic systems on a closed, connected, oriented manifold $M$. If ${\rm MLS}(g_1,\sigma_1) = {\rm MLS}(g_2,\sigma_2)$, $\mathrm{Vol}(g_1)=\mathrm{Vol}(g_2)$, and $[\sigma_1] = [\sigma_2]$, is it true that $g_1$ and $g_2$ must be isometric?
	\end{question}
	
	As evidence towards a positive answer for this question, suppose that $g_1$ and $g_2$ are two Anosov metrics with non-positive sectional curvature. If $(g_1, \sigma_1)$ is an Anosov magnetic system satisfying $\MLS(g_1, \sigma_1) = \MLS(g_2)$ and $\text{Vol}(g_1) = \text{Vol}(g_2)$, then it follows from the fact that geodesics are the unique length minimizers in their free homotopy class that $\MLS(g_1) \leq \MLS(g_2)$. As long as the metrics are sufficiently close according to \cite[Theorem 2]{guillarmou2019marked}, we have that $g_1$ and $g_2$ are isometric and $\sigma_1 \equiv 0$.\footnote{Note that one can also use this argument along with \cite[Theorem 1.1]{crokedairbekov} to deduce Corollary ~\ref{cor:main2} in the case where $g_1$ and $g_2$ are negatively curved metrics.}
	
	\subsection*{Acknowledgments}
	The authors would like to thank Andrey Gogolev, Javier Echevarr\'{i}a Cuesta, Gabriel Paternain, and Gabrielle Benedetti for many useful discussions throughout this project.
	
	\section{Preliminaries} \label{sec:prelim}
	
	Throughout, let $M$ be a closed, connected, oriented surface, let $(g,b)$ be a magnetic system on $M$, and let $\pi_g : S_gM \rightarrow M$ be the footprint map for the unit tangent bundle. The \emph{geodesic vector field} $X^{g}$ is the infinitesimal generator of the geodesic flow $\varphi_g^t$ on $S_{g}M$. Since $M$ is oriented, we may consider the \emph{rotation flow} $\rho^t_g : S_gM \rightarrow S_gM$, given by $\rho^t_g(x,v) \coloneqq (x, e^{it}v)$. Its infinitesimal generator is the \emph{vertical vector field}, which we denote by $V^g$. Finally, the \emph{horizontal vector field} on $S_gM$ is the vector field given by $H^g \coloneqq [V^g, X^g]$. Observe that the vector fields $X^g,H^g,$ and $V^g$ form a global frame on $S_gM$, which we refer to as \emph{Cartan's moving frame}. They satisfy the relations
	\begin{equation} \label{eqn:cartan1}
		[V^g,H^g] = -X^g, \quad [V^g,X^g] = H^g, \quad [X^g,H^g] = \pi_g^*(K^g) V^g,
	\end{equation}
	where $K^g$ is the Gaussian curvature on $M$. Dual to these vector fields are the $1$-forms $\alpha^g, \beta^g,$ and $\psi^g$ on $S_gM$, which we refer to as \emph{Cartan's moving coframe}. Using \eqref{eqn:cartan1}, one can show that they satisfy the relations
	\begin{equation} \label{eqn:cartan2}
		{\mathrm d}\alpha^g = \psi^g \wedge \beta^g, \quad {\mathrm d}\beta^g = -\psi^g \wedge \alpha^g, \quad {\mathrm d}\psi^g = -\pi_g^*(K^g) \alpha^g \wedge \beta^g.
	\end{equation}
	Notice that $\pi_g^*(\Omega_g) = \alpha^g \wedge \beta^g$, hence we can rewrite the last relation as
	\begin{equation} \label{eqn:psi}
		{\mathrm d}\psi^g = -\pi_g^*(K^g \Omega_g).
	\end{equation}
	Moreover, the Liouville volume form $\mu^g$ satisfies $\mu^g = \alpha^g \wedge \beta^g \wedge \psi^g$.
	
	Next, observe that the flow $\varphi_{g,b}^t$ is a Hamiltonian flow, induced by the function \hbox{$H : TM \rightarrow \R$} given by $H(x,v) \coloneqq \frac{1}{2} g_x(v,v)$ and the symplectic form given by
	\begin{equation} \label{eqn:symplectic}
		\omega_{g,b} \coloneqq - {\mathrm d}\alpha_g + \pi_g^*(b \Omega_g).
	\end{equation}
	With this, we now have the tools to prove the following.
	
	\begin{lem} \label{lem:reduction}
		Let $(g_1, b_1)$ and $(g_2, b_2)$ be Anosov magnetic systems
		on $M$, smoothly conjugate via $h : S_{g_1}M \rightarrow
		S_{g_2}M$; then:
		\begin{align*}
			h^*(\mu^{g_2}) = \frac{{\rm Area}(g_2)}{{\rm Area}(g_{1})}\mu^{g_1}.
		\end{align*}
	\end{lem}
	
	\begin{proof}
		Notice that there exists $J_h \in C^\infty(S_{g_1}M)$ so that $h^*(\mu^{g_2}) = J_h \mu^{g_1}$. Since the magnetic flow is Hamiltonian, it preserves the Liouville volume form, and thus
		\[ J_h \mu^{g_1} = h^*(\mu^{g_2}) = h^*((\varphi_{g_2,b_2}^t)^*(\mu^{g_2})) = (\varphi_{g_1,b_1}^t)^* (h^*(\mu^{g_2})) = (J_h \circ \varphi_{g_1,b_1}^t) \mu^{g_1}.\]
		We deduce that $J_h$ is constant along orbits. Since the magnetic flow is Anosov, it admits a dense orbit, and thus $J_h$ must be constant everywhere. Next, recall that
		\[ \int_{S_{g_i}M} \mu^{g_i} = 2\pi \text{Area}(g_i).\]
		Using a change of variables, we have
		\[ J_h \int_{S_{g_1}M} \mu^{g_1} = \int_{S_{g_1}M} h^*(\mu^{g_2}) = \int_{S_{g_2}M} \mu^{g_2}. \]
		Thus, $J_h$ is the ratio of the areas.
	\end{proof}
	
	Using the Hamiltonian property, we can also deduce that the infinitesimal generator of the magnetic flow can be written as $X_{g,b} = X_g + \pi_g^*(b) V_g$, implying that
	\begin{equation} \label{eqn:symplectic_volume}
		\iota_{X_{g,b}} \mu^g = \omega_{g,b}.
	\end{equation}
	
	Next, since $b \Omega_g$ is a closed $2$-form on $M$, we may use the isomorphism $H^2(M,\R) \cong \R$ to get
	\begin{equation} \label{eqn:cohomology}
		b \Omega_g = c_{g,b} K^g \Omega_g + {\mathrm d}\theta,
	\end{equation}
	where $\theta$ is a $1$-form on $M$ and
	\begin{equation} \label{eqn:constant} c_{g,b} \coloneqq \frac{1}{2\pi \chi(M)} \int_M b \Omega_g\end{equation}
	In particular, we see that $\pi_g^*(b \Omega_g) = {\mathrm d}(-c_{g,b} \psi_g + \pi_g^*(\theta))$, and thus if \begin{equation} \label{eqn:primitive} \tau_{g,b} \coloneqq -\alpha_g - c_{g,b} \psi_g + \pi_g^*(\theta), \end{equation}
	then $\omega_{g,b} = {\mathrm d}\tau_{g,b}$ on $S_{g}M$. Using \cite[Lemma 7.1]{solly}, this shows that the magnetic flow is \emph{homologically full}, meaning that every integral homology class on $S_gM$ has a closed $(g,b)$-geodesic representative $(\gamma, \dot{\gamma})$.\footnote{There are many other ways to deduce this result. As an alternative, one can use \cite{ghys1984flots} to deduce that the magnetic flow is orbit equivalent to an Anosov geodesic flow, hence it is homologically full.} It is well-known that this implies that for all closed $1$-forms $\eta$ on $S_{g_1}M$, we have
	\begin{equation} \label{eq:injection-eta} \int_{S_{g_1}M} \eta(X_{g_1,b_1}) \mu^{g_1} = 0;\end{equation}
	see, for instance, \cite{sharp} and \cite[Theorem 2.4]{gogolev2020abelian}. Furthermore, this primitive is a useful tool in showing that the existence of a smooth volume preserving conjugacy implies that $[b_1 \Omega_{g_1}] = \pm [b_2 \Omega_{g_2}]$.
	
	\begin{lem}[{\cite[Proposition 4.5]{javier}, \cite[Lemma
			4.1]{paternain2006magnetic}}] \label{lem:homology_class} Let
		$(g_1, b_1)$ and $(g_2, b_2)$ be Anosov magnetic systems on $M$
		which are smoothly conjugate; then
		\begin{align*}
			A_1 \left[1 + \frac{2 \pi \chi(M) c_{g_2,b_2}^2}{A_2} \right] = A_2 \left[1 + \frac{2 \pi \chi(M) c_{g_1,b_1}^2}{A_1} \right] 
		\end{align*}
	\end{lem}
	
	\begin{proof}
		As before, let ${h : S_{g_1}M \rightarrow S_{g_2}M}$ denote the
		smooth conjugacy of magnetic flows, and introduce the notation $A_{i} = \text{Area}(g_{i})$. 
		Lemma~\ref{lem:reduction} and \eqref{eqn:symplectic_volume}
		imply that
		$h^*(\omega_{g_2,b_2}) = (A_2/A_1)\omega_{g_1,b_1}$,
		hence ${h^*(\tau_{g_2,b_2}) - (A_2/A_1)\tau_{g_1,b_1}}$ is a closed $1$-form on $S_{g_1}M$. Thus, using ~\eqref{eq:injection-eta}, we have
		\[ \frac{A_1}{A_2}\int_{S_{g_2}M} \tau_{g_2,b_2}(X_{g_2,b_2}) \mu^{g_2} =
		\frac{A_{2}}{A_{1}}\int_{S_{g_1}M}\tau_{g_1,b_1}(X_{g_1,b_1})
		\mu^{g_1}.\]  Let $\theta_i$ be
		$1$-forms on $M$ satisfying \eqref{eqn:primitive} for
		$b_i \Omega_{g_i}$. Notice that we have
		$\tau_{g_i,b_i}(X_{g_i,b_i}) = - 1 - c_{g_i,b_i} \pi_{g_i}^*(b_i) +
		\theta_i$, and thus
		\[ \frac{A_1}{A_2} \left[ - A_2 - c_{g_2,b_2} \int_M b_2 \Omega_{g_2} \right] = \frac{A_2}{A_1} \left[  - A_1 - c_{g_1,b_1} \int_M b_1 \Omega_{g_1}\right]. \]
		Applying~\eqref{eqn:constant} yields the desired result.
	\end{proof}
	
	This primitive is also a useful tool for establishing the following relation.
	
	\begin{lem} \label{lem:primitive_identity}
		Let $(g_1, b_1)$ and $(g_2,b_2)$ be Anosov magnetic systems on $M$ which are smoothly conjugate via the map $h : S_{g_1}M \rightarrow S_{g_2}M$ and which satisfy ${\rm Area}(g_1) = {\rm Area}(g_2)$. There is a closed $1$-form $\omega$ on $M$ such that if $\gamma_1$ is a closed orbit for $(g_1, b_1)$ and $\gamma_2$ is the corresponding closed orbit for $(g_2,b_2)$ under $h$, then we have
		\[ \pm c_{g_2,b_2}\int_{\gamma_2} b_2 + \int_{\gamma_2} \theta_2 = - c_{g_1,b_1} \int_{\gamma_1} b_1 + \int_{\gamma_1} [\theta_1 + \omega].  \]
	\end{lem}
	
	\begin{proof}
		Using the Gysin sequence (see \cite{bott2013differential}), one can show that the pullback ${\pi_{g_1}^* : H^1(M,\R) \rightarrow H^1(S_{g_1}M, \R)}$ is an isomorphism \cite[Corollary 8.10]{merry2011inverse}. As a consequence, there is a smooth function $u : S_{g_1}M \rightarrow \R$ and a closed $1$-form $\omega$ on $M$ so that $h^*(\tau_{g_2,b_2}) - \tau_{g_1,b_1} = \pi_{g_1}^*(\omega) + du$. Contracting both sides of this equation against $X_{g_1,b_1}$ and integrating over the orbit $\gamma_1$, the result follows from ~\eqref{eq:injection-eta}.
	\end{proof}
	
	Another characterization of $(g,b)$-geodesics is that they are minimizers of a functional that behaves like the length functional. More precisely, let $(g,b)$ now be an Anosov magnetic system on $M$. Inside of every non-trivial free homotopy class $\nu \in \pi_1(M)$ there exists a unique $(g,b)$-geodesic $\gamma_\nu \in \nu$ \cite{contreras2000palais}. Given a closed curve $\gamma :S^1 \rightarrow M$ in the free homotopy class $\nu$, we define $\Sigma_{g,b}(\gamma)$ to be a $2$-chain on $M$ with boundary given by $\gamma - \gamma_\nu$. With this, we can define the \emph{magnetic length} of an arbitrary closed curve $\gamma$ to be
	\begin{equation} \label{eqn:magnetic_length}
		L_{g,b}(\gamma) \coloneqq \ell_g(\gamma) + \int_{\Sigma_{g,b}(\gamma)} b \Omega_g,
	\end{equation}
	where $\ell_g$ denotes the (Riemannian) length of the curve with respect to the metric $g$. Notice that $L_{g,b}$ is well-defined by \cite[Lemma 2.2]{merry2010closed} and, in particular, we have the following.
	
	\begin{lem}[{\cite{bahri1998periodic, merry2010closed}}] \label{lem:magnetic_length}
		Let $(g,b)$ be an Anosov magnetic system on $M$. If  $\gamma : S^1 \rightarrow M$ is an arbitrary closed curve inside of a free homotopy class $\nu$, then
		\[ L_{g,b}(\gamma_\nu) \leq L_{g,b}(\gamma).\]
	\end{lem}

	Finally, it will be convenient to work with an abstract circle bundle which is independent of the choice of metric. Namely, let $SM$ be the principal circle bundle over $M$ with $p : SM \rightarrow M$ is the projection map, and the fibers are given by $S_xM = (T_xM \smallsetminus \{0\}) / \sim$, where $v \sim w$ if and only if $v = Cw$ with $C > 0$. Given any Riemannian metric $g$ on $M$, we can identify $S_gM$ with $SM$ by sending vectors to their equivalence classes. When working on $SM$, we will use the same symbol to identify all objects from $S_gM$ identified in $SM$.
	
	\section{Proof of Theorem ~\ref{thm:main1}}
	
	The goal of this section is to prove the following theorem.
	
	\begin{thm} \label{thm:conformal}
		Let $M$ be a closed, connected, oriented surface, and let ${\rho \in C^\infty(M, \R^+)}$. Suppose $(g, b_1)$ and $(\rho g, b_2)$ are two Anosov magnetic systems on $M$ satisfying the following conditions:
		\[ {\rm MLS}(g,b_1) = {\rm MLS}(\rho g,b_2), \ {\rm Area}(g) = {\rm Area}(\rho g), \ \mbox{and} \ [b_1 \Omega_{g}] = [b_2 \Omega_{g}].\]
		Then $\rho \equiv 1$ and $b_1 = b_2$.
	\end{thm}
	
	Observe that \cite[Theorem 1.1]{javier} reduces Theorem
	~\ref{thm:main1} to this setting, and thus it is sufficient to work
	in this conformal setting. The proof that we present
	below follows, to some extent, the strategy of the proof
	in ~\cite[Theorem 2]{katok} for Riemanian geodesics. In order to adapt the argument
	presented there to our setting, we need to compare the
	magnetic length functional defined above with the
	Riemannian length functional. This comparison requires the
	lemma below, which allows us to bound the difference between the two
	functionals. Let $h : SM \rightarrow SM$ denote the conjugacy
	described in the Introduction on the level of $SM$, and note that if
	$[b_1 \Omega_{g_1}] = [b_2 \Omega_{g_2}]$, then
	$b_2 \Omega_{g_2} - b_1 \Omega_{g_1} = {\mathrm d} \zeta$ for some $1$-form
	$\zeta$ on $M$. 
	
	\begin{lem} \label{lem:technical}
		Let $M$ be a closed, connected, oriented surface. Suppose $(g_1,
		b_1)$ and $(g_2, b_2)$ are two Anosov magnetic systems on $M$
		satisfying the following properties:
		\[ {\rm MLS}(g_1,b_1) = {\rm MLS}(g_2,b_2), \ {\rm Area}(g_1) = {\rm Area}(g_2), \ \mbox{and} \ [b_1 \Omega_{g_1}] = [b_2 \Omega_{g_2}]\]
		Then for every homologically trivial closed $(g_1,b_1)$-geodesic $\gamma_1$, we have
		\begin{equation} \label{eqn:step4} \int_{\Sigma_{g_2,b_2}(\gamma_1)} b_2 \Omega_{g_2} = \int_{\gamma_1} \zeta. \end{equation}
		
	\end{lem}

	\begin{proof}
		As in \eqref{eqn:primitive}, write $b_1 \Omega_{g_1} = c_{g_1,b_1} K^{g_1} \Omega_{g_1} + {\mathrm d}\theta_1$, and define the $1$-form ${\kappa \coloneqq - c_{g_1,b_1} \psi^{g_1} + p^*(\theta_1)}$ on $SM$. Observe that \hbox{$p^*(b_2 \Omega_{g_2}) = d(\kappa + p^*(\zeta))$}. In particular, if $\gamma_1$ is a closed $(g_1,b_1)$-geodesic and $\tilde{\Sigma}_{g_2,b_2}(\gamma_1)$ is the lift of the $2$-chain $\Sigma_{g_2,b_2}(\gamma_1)$ from $M$ to $SM$, then Stokes' theorem yields
		\begin{equation*} \label{eqn:step1} \begin{split} \int_{\Sigma_{g_2,b_2}(\gamma_1)} b_2 \Omega_{g_2} & = \int_{\tilde{\Sigma}_{g_2,b_2}(\gamma_1)} p^*(b_2 \Omega_{g_2}) \\&=  \left[ \int_{\gamma_1}  \kappa(X_{g_1,b_1}) - \int_{\gamma_2} \kappa(X_{g_2,b_2}) \right] + \left[\int_{\gamma_1} \zeta - \int_{\gamma_2} \zeta \right], \end{split}\end{equation*}
		where $\gamma_{2}$ is the image of $\gamma_{1}$ by $h$, \ie
		the $(g_{1},b_{1})$-geodesic in the same homotopy class as $\gamma_{2}$.
		We can then rewrite this as
		\begin{equation} \label{eqn:step2}
			\int_{\Sigma_{g_2,b_2}(\gamma_1)} b_2 \Omega_{g_2} =
			\int_{\gamma_2} [h_*(\kappa)(X_{g_2,b_2}) - \kappa(X_{g_2,b_2})]
			+ \left[\int_{\gamma_1} \zeta - \int_{\gamma_2} \zeta
			\right].
		\end{equation}
		Next, using Lemma ~\ref{lem:reduction} and
		\eqref{eqn:symplectic_volume}, we deduce that
		$h_*(\omega_{g_1,b_1}) = \omega_{g_2,b_2}$. Since, by
		definition,
		$\omega_{g_i, b_i} = - {\mathrm d}\alpha^{g_i} + p^*(b_i \Omega_{g_i})$,
		we observe that the $1$-form
		${- \alpha^{g_2} + \kappa + p^*(\zeta) - h_*(-\alpha^{g_1} +
			\kappa)}$ on $SM$ is closed. Using the fact that $h$
		is a conjugacy, so $h_*(X_{g_1,b_1}) = X_{g_2,b_2}$, it follows that
		\[ \iota_{X_{g_2,b_2}}(- \alpha^{g_2} + \kappa + p^*(\zeta) -
		h_*(-\alpha^{g_1} + \kappa)) = \iota_{X_{g_2,b_2}}(\kappa -
		h_*(\kappa)) + \zeta.\] Since $\gamma_2$ is also
		homologically trivial, the integral of the left hand side
		over the lift of $\gamma_2$ is zero, hence
		\begin{equation} \label{eqn:step3} \begin{split} 0 = \int_{\gamma_2}  [\kappa(X_{g_2,b_2}) - h_*(\kappa)(X_{g_2,b_2})]
				+ \int_{\gamma_2} \zeta.
		\end{split} \end{equation}
		Substituting \eqref{eqn:step3} into \eqref{eqn:step2}, the result follows.
	\end{proof}

	\begin{rem}
		Notice that if we have $[b_1 \Omega_{g_1}] = -[b_2 \Omega_{g_2}]$, so ${b_2 \Omega_{g_2} + b_1 \Omega_{g_1} = {\mathrm d}\zeta}$, then \eqref{eqn:step3} becomes
		\[ 0 = -\int_{\gamma_2} [\kappa(X_{g_2,b_2}) + h^*(\kappa)(X_{g_2,b_2})]
		+ \int_{\gamma_2} \zeta.  \]
		Substituting this into \eqref{eqn:step2} yields
		\[ \int_{\Sigma_{g_2,b_2}(\gamma_1)} b_2 \Omega_{g_2} =  -2 \int_{\gamma_2} \kappa(X_{g_2,b_2}) + \int_{\gamma_1} \zeta.\]
		The term $\kappa(X_{g_2,b_2})$ presents a challenge in applying the upcoming ergodic arguments, and is why we need our cohomology class assumption.
	\end{rem}
	
	With Lemma ~\ref{lem:technical} in hand, we can now prove the result.
	
	\begin{proof}[Proof of Theorem ~\ref{thm:conformal}]
		To ease our notation, let
		$A \coloneqq \text{Area}(g) = \text{Area}(\rho g)$, and let
		$g_1 \coloneqq g$ and $g_2 \coloneqq \rho g_1$.  Jensen's inequality
		yields:
		\begin{align*}
			\frac{1}{A} \int_M \rho^{1/2} \Omega_{g_1} \leq \left[
			\frac{1}{A} \int_M \rho \Omega_{g_1} \right]^{1/2} = 1,
		\end{align*}
		with equality if and only if $\rho \equiv 1$.  Suppose for
		contradiction that $\rho \neq 1$ and let $\epsilon > 0$ be
		small enough so that
		\begin{align*}
			\frac{1}{A} \int_M \rho^{1/2} \Omega_{g_1} + \epsilon < 1.
		\end{align*}
		
		Let ${\mathrm d}\lambda_{g_i}$
		be the normalized Liouville measure associated to the Liouville
		volume form $\mu^{g_i}$. Since the flow $\varphi_{g_1,b_1}^t$ is
		ergodic with respect to this measure, Birkhoff's Ergodic Theorem
		implies that there exists $v \in S_{g_1}M$ so that for every
		continuous function $f \in C^0(S_{g_1}M)$,
		\begin{align*}
			\lim_{T \rightarrow \infty} \frac{1}{T} \int_0^T
			f(\varphi_{g_1,b_1}^t(v)) {\mathrm d}t = \int_{S_{g_1}M} f {\mathrm d}\lambda_{g_1}.
		\end{align*}
		We apply the above to $f = (\pi^{*}_{g_{1}}\rho)^{1/2}$; thus,
		there exists a $T_0 > 0$ so that for all $T \geq T_0$, we have
		\begin{align*}
			\frac{1}{T} \int_0^T (\pi_{g_{1}}^{*}\rho)^{1/2}(\varphi_{g_1,b_1}^t(v)) {\mathrm d}t <
			\frac{1}{A} \int_M \rho^{1/2} \Omega_{g_1} + \frac{\epsilon}{2}.
		\end{align*}
		Furthermore, taking $T_0$ larger if necessary, we may also assume that
		for all $T \geq T_0$ we have
		\begin{align*}
			\frac{1}{T} \int_0^T \zeta(\varphi_{g_1,b_1}^t(v)) {\mathrm d}t <
			\int_{S_{g_2}M} \zeta {\mathrm d}\lambda_{g_2} + \frac{\epsilon}{2} =
			\frac{\epsilon}{2}.
		\end{align*}
		Since the orbit of $v$ is dense, we can pick $T$ to be large enough so that $\varphi^{T}_{g_{1},b_{1}}$ is arbitrarily close to $v$. Using
		density of homologically trivial orbits
		\cite[Lemma 2.7]{gogolev2020abelian}, this implies that there exists a closed
		homologically trivial $(g_1,b_1)$-geodesic $\gamma_1$ such that
		\begin{equation} \label{eqn:first_ineq}
			\frac{1}{\ell_{g_1}(\gamma_1)} \int_{\gamma_1} \rho^{1/2} <
			\frac{1}{A} \int_M \rho^{1/2} \Omega_{g_1} +
			\frac{\epsilon}{2}\end{equation} and
		\begin{equation} \label{eqn:second_ineq}
			\left|\frac{1}{\ell_{g_1}(\gamma_1)} \int_{\gamma_1}
			\zeta\right| < \frac{\epsilon}{2}.\end{equation} Rewriting
		\eqref{eqn:first_ineq} using the fact that the metrics are conformally
		related, we see that
		\begin{equation} \label{eqn:third_ineq}
			\frac{\ell_{g_2}(\gamma_1)}{\ell_{g_1}(\gamma_1)} < \frac{1}{A} \int_M
			\rho^{1/2} \Omega_{g_1} + \frac{\epsilon}{2}. \end{equation} Let
		$\gamma_2$ be the corresponding closed $(g_2,b_2)$-geodesic in the
		same free homotopy class as $\gamma_1$. Utilizing Lemma
		~\ref{lem:magnetic_length}, we observe
		\[ \ell_{g_2}(\gamma_2) \leq \ell_{g_2}(\gamma_1) +
		\int_{\Sigma_{g_2,b_2}(\gamma_1)} b_2 \Omega_{g_2}.\] This, along with
		Lemma ~\ref{lem:technical}, \eqref{eqn:first_ineq}, and
		\eqref{eqn:second_ineq}, implies that
		\[\begin{split} 1 =
			\frac{\ell_{g_2}(\gamma_2)}{\ell_{g_1}(\gamma_1)} &\leq
			\frac{\ell_{g_2}(\gamma_1)}{\ell_{g_1}(\gamma_1)} +
			\frac{1}{\ell_{g_1}(\gamma_1)}\int_{\Sigma_{g_2,b_2}(\gamma_1)} b_2
			\Omega_{g_2} \\&\leq \frac{\ell_{g_2}(\gamma_1)}{\ell_{g_1}(\gamma_1)}
			+ \frac{1}{\ell_{g_1}(\gamma_1)}\int_{\gamma_1}\zeta \\ & <
			\frac{1}{A} \int_M \rho^{1/2} \Omega_{g_1} + \epsilon \\ & <
			1.\end{split}\] We have reached a contradiction, and thus $\rho \equiv
		1$.
		
		Given
		$g_1 = g_2 \eqqcolon g$, we now need to show that $b_1 = b_2$. Let
		$\gamma_1$ be a closed $(g_1,b_1)$-geodesic, and let $\gamma_2$ be the
		corresponding $(g_2,b_2)$-geodesic in the same free homotopy
		class. Lemma ~\ref{lem:magnetic_length} yields ${\ell_{g_2}(\gamma_2) = L_{g_2,b_2}(\gamma_2) \leq L_{g_1,b_1}(\gamma_1)}$, and the marked length spectrum
		assumption implies
		\[
		0 \leq \int_{\Sigma_{g_2, b_2}(\gamma_1)} b_2 \Omega_{g_2}. \] 
		Using
		the Gauss-Bonnet theorem along with Stokes' theorem and \eqref{eqn:primitive}, we can rewrite
		this as
		\[
		0 \leq c_{g_2,b_2} \int_{\gamma_2} b_2 - c_{g_2,b_2} \int_{\gamma_1}
		b_1 + \int_{\gamma_1} \theta_2 - \int_{\gamma_2} \theta_2.\] Lemma
		~\ref{lem:primitive_identity} along with the cohomology class
		assumption implies that there is a closed $1$-form $\omega$ on $M$ such that
		\[
		0 \leq \int_{\gamma_1} [\theta_2 - \theta_1 - \omega].\]
		The choice of $\gamma_1$ was
		arbitrary, hence this hold for all closed
		$(g_1,b_1)$-geodesics. The non-positive Livshits theorem \cite[Theorem
		1]{lopes2005sub} implies that there are H\"{o}lder continuous
		functions $F : S_{g_1}M \rightarrow [0,\infty)$ and $V : S_{g_1}M
		\rightarrow \R$ such that $\theta_2 - \theta_1 - \omega =
		X_{g_1,b_1}(V) + F$. Integrating this over $S_{g_1}M$ with respect to
		the normalized Liouville measure, we have
		\[
		0 = \int_{S_{g_1}M} [\theta_2 - \theta_1 - \omega] {\mathrm d}\lambda_{g_1} =
		\int_{S_{g_1}M} F{\mathrm d}\lambda_{g_1}. \] 
		Thus, $F \equiv
		0$. Using \cite[Theorem B]{paternain2005longitudinal}, we see that
		$\theta_2 - \theta_1 - \omega$ is an exact $1$-form on $M$, hence
		$\theta_2 - \theta_1$ is closed. Finally, observe that $(b_2 - b_1)
		\Omega_g = {\mathrm d}(\theta_2 - \theta_1) = 0$ by \eqref{eqn:primitive}, hence
		$b_2 = b_1$.
	\end{proof}
	
	\section{Examples} \label{section:examples}
	
	The goal of this section is to construct a series of examples in order to better understand the assumptions in Theorem ~\ref{thm:main1}, as well as to understand how Corollary ~\ref{cor:main} compares to known results on marked length spectrum rigidity. As a first step, we show how the area assumption cannot be removed.
	
	\begin{lem} \label{lem:area}
		Let $M$ be a closed, connected, oriented surface, and let $g$ be a hyperbolic metric on $M$. We have ${\MLS(g,b) = (1-b^2)^{-1/2} \MLS(g)}$ for every constant $|b| < 1$.
	\end{lem}
	
	In particular, consider the homothetic metric $g_b \coloneqq (1-b^2)^{-1} g$. Lemma ~\ref{lem:area} shows that $\MLS(g_b) = \MLS(g,b)$ and $\text{Area}(g) \neq \text{Area}(g_b)$, and thus the area assumption in Theorem ~\ref{thm:main1} is necessary.
	
	\begin{proof}[Proof of Lemma ~\ref{lem:area}]
		We lift the magnetic system to the universal cover $\mathbb{H} = \{(x,y)\in \mathbb{R}^2, \ y>0 \}$. Denote the lifted magnetic system on $\mathbb{H}$ by $(\tilde{g}, b)$, where $\tilde{g}$ is given by $\tilde{g} = y^{-1} ({\mathrm d}x^2 + {\mathrm d}y^2)$. In particular, $M$ can viewed as the quotient of $\mathbb{H}$ under the action of a cocompact lattice $\Gamma$ in $\text{PSL}(2,\R)$.
		
		Let $\nu \in\pi_1(M)$, and let $\gamma:\R \to M$ the unique closed $g$-geodesic representative. 
		Without loss of generality we can assume that the lifted geodesic $\tilde{\gamma}$ is of the form $\tilde{\gamma}(\R)=\{(0,r) \in \R^2 \ | \  r> 0\}$. In particular, there exists an isometry $f\in \Gamma$ so that $f(\tilde{\gamma}(t)) = \tilde{\gamma}(t+\ell) $ for every $t \in \R$ (see, for example, \cite{Farb}). Since $f$ fixes $(0,1)$ and the point at infinity, we deduce that $f(x,y)=e^{\ell_g(\gamma)} (x,y)$, thus
		\[ \ell_g(\gamma) = \int_1^{e^{\ell_g(\gamma)}} \frac{\mathrm{d}y}{y}. \]
		Let $\gamma_b: \R \to M$ be the unique closed $(g,b)$-geodesic with free homotopy class $\nu$. Notice that we have $\tilde{\gamma}_b(\R)=\{\left( rb, r(1-b^2)\right) \ | \ r>0 \}$ and ${f( b, 1-b^2) = \left( e^{\ell_g(\gamma)}b, e^{\ell_g(\gamma)}(1-b^2)\right)}$, hence
		\[ \ell_g(\gamma_b) = \int_{1}^{e^{\ell_g(\gamma)}} \frac{\mathrm{d}y}{y\sqrt{1-b^2}} = \frac{\ell_g(\gamma)}{\sqrt{1-b^2}}.\]
		The result follows. \qedhere
	\end{proof}
	
	
	
	\begin{figure}[H]
		\centering
		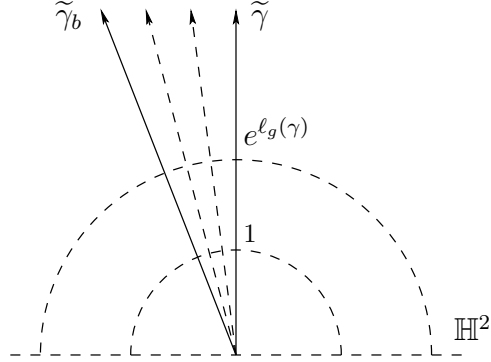
		\caption{An example showing how $\tilde{\gamma}$ moves as $b$ changes.}
		\label{fig:magnetic_3}
	\end{figure}

	Next, we show how the assumption on the cohomology classes in Theorem ~\ref{thm:main1} cannot be removed.
	
	\begin{lem} \label{lem:intensity}
		Let $M$ be a closed, connected, oriented surface, and let $g$ be a hyperbolic metric on $M$. We have ${\MLS(g,b) = \MLS(g,-b)}$ for every constant $|b| < 1$.
	\end{lem}

	\begin{proof}[Proof of Lemma ~\ref{lem:intensity}]
		Define
		\[ X^{g,b} \coloneqq \frac{1}{2}\begin{pmatrix} 1 & b \\ -b & -1 \end{pmatrix}. \]
		It is well-known that there is a cocompact lattice $\Gamma$ in $\text{PSL}(2,\R)$ such that the magnetic flow associated to $(g,b)$ on $S_gM$ can be interpreted as the flow on $\Gamma \backslash \text{PSL}(2,\R)$ given by $\phi^{g,b}_t(\Gamma q) \coloneqq \Gamma q e^{tX^b}$ (see, for example, \cite{paternain2006magnetic}). Observe that the matrix $c \in \text{SL}(2,\R)$ given by
		\[ c \coloneqq \frac{1}{1-b^2} \begin{pmatrix} 1 & -b \\ -b & 1 \end{pmatrix}\]
		satisfies the condition that $X^{g,b} c = c X^{g,-b}$. Thus, an easy argument using induction and the definition of the exponential map tells us that if we define the map $h : \Gamma \backslash \text{PSL}(2,\R) \rightarrow \Gamma \backslash \text{PSL}(2,\R)$ by $h(\Gamma q) \coloneqq \Gamma q c,$
		then we have $h \circ \phi_{g,b}^t = \phi_{g,-b}^t \circ h$. This is a smooth conjugacy between the magnetic flows for $(g,b)$ and $(g,-b)$. Furthermore, if we define a family of matrices $c_s \in \text{SL}(2,\R)$ by
		\[ c_s \coloneqq \frac{1}{1-sb^2} \begin{pmatrix} 1 & -sb \\ -sb & 1 \end{pmatrix}, \]
		then we have an associated family of maps $h_s : \Gamma \backslash \text{PSL}(2,\R) \rightarrow \Gamma \backslash \text{PSL}(2,\R)$ given by  $h_s(\Gamma q) \coloneqq \Gamma q c_s$.
		Observe that the family is continuous in $s$ and satisfies $h_0 = \text{Id}$ and $h_1 = h$, so $h$ is homotopic to the identity. As discussed in \cite[Section 10.3]{burns1985manifolds} and \cite[Section 3]{croke1990rigidity}, the subgroup of $\pi_1(SM)$ generated by the kernel of the footprint map equals the center of $\pi_1(SM)$.
		Since $h$ is invertible, we have that $h_*$ is an isomorphism on the level of the fundamental groups, hence induces an isomorphism on the level of $\pi_1(M)$. This shows that $\MLS(g,b) = \MLS(g,-b)$.
	\end{proof}
	
	Finally, we construct an example of a closed, oriented, connected surface such that there exists an exact Anosov magnetic system $(g,b)$ where the metric $g$ is not Anosov by modifying the arguments in \cite[Section 7]{burns2002anosov}. This shows that it is not necessary to assume that $g$ is Anosov in Corollary ~\ref{cor:main}. We recall the following lemma, and briefly sketch its proof. For more details, we point the reader to \cite[Sections 6 and 7]{burns2002anosov}.
	
	\begin{lem} \label{lem:burns_example}
		There exists a closed, connected, orientable surface $M$ and an Anosov magnetic system $(g, b)$ on $M$ with $b$ constant such that $g$ is not Anosov.
	\end{lem}
	
	\begin{proof}[Sketch of Proof]
		
		We will need the following.
		
		\begin{lem}[{\cite[Lemma 6.1]{burns2002anosov}}] \label{lem:exp_dichotomy}
			Let $(g,b)$ be a magnetic system on $M$. If there are constants $T > 0$ and $H > 1$ such that for any $(g,b)$-geodesic $\gamma$ and any solution $u$ to the \emph{magnetic Riccati equation}
			\begin{equation} \label{eqn:riccati} \dot{u}(t) + u^2(t) + K^{g,b}(\gamma(t), \dot{\gamma}(t)) = 0\end{equation}
			with $u(0) \geq 0$, we have $1/H \leq u(T) \leq H$, then the magnetic system is Anosov.
		\end{lem}
		
		Consider a rotationally symmetric surface $N=\R \times S^1$ with coordinates $(s,\theta)$. For a smooth function $u:\R \to \R$, equip $N$ with the following Riemannian metric:
		\[g_u = \mathrm{d}s^2 + \mathrm{exp}\left(\int_0^s u(\tau)\mathrm{d}\tau \right)^{2}\mathrm{d}\theta^2.\]
		Note that the Gaussian curvature is given by the formula 
		\[{K^{g_u}(s,\theta) =  -u'(s) - u^2(s)},\] 
		where $'$ denotes the derivative with respect to the $s$ coordinate. For $s\in \R$ the unit speed parallel 
		\[\alpha_s^{\pm}(t) = \left(s, \pm \mathrm{exp}\left(\int_0^s u(\tau)\mathrm{d}\tau \right)^{-1} t \right)\] 
		has $g_u$-geodesic curvature $\mathrm{k}^{g_u}(t)=\pm u(s) $. In particular, $\alpha_s^\pm$ is a closed $g_u$-geodesic if and only if $u(s)=0$ \cite[Lemma 6.2]{burns2002anosov}. For $\delta > 0$ small, consider the interval $[-\delta, \delta] \subset [-1/4,1/4]$. Let $u(s) = \mathrm{tanh}(s)$, except for a small $C^0$ perturbation in the interval $[-\delta, \delta]$ where the function $u$ satisfies the following:
		\begin{itemize}
			\item $|u(s)| < \mathrm{tanh}(1/4)$ and $K^{g_u}(s) < 1/4$, $\forall s \in [-\delta, \delta]$.
			\item $u(0) = 0$ and $u'(0)<0$.
		\end{itemize}
		With this choice of $u$, we can immediately deduce that 
		$\alpha_0^\pm(t) = (0,t)$ is a closed $g_u$-geodesic and $K^{g_u}(\gamma_0(t)) >0$. 
		
		Denote by $(M,\bar{g}_u)$ the Riemannian surface obtained through a standard compactification of $(N,g_u)$ -- see, for instance, \cite{ballmann1987surfaces}. By construction, the metric $\bar{g}_u$ carries a closed geodesic contained in a region with positive curvature, hence the geodesic flow cannot be Anosov. Consider now the magnetic system $(g_u, 1/2)$ on $N$, and observe that the magnetic curvature is given by $K^{g_u, 1/2}(s) = K^{g_u}(s) + 1/4$. In particular, note that $K^{g_u, 1/2}$ is at most $1/2$ in the region $C_{\delta} = [-\delta,\delta] \times S^1$, and at most $-3/4$ outside $C_{\delta}$. Observe also that, due to the choice of $u$, for $s\in [-1/4,1/4]$ the $g_u$-geodesic curvature of a parallel is bounded in absolute value by $1/4$ and unit speed $(g_u, 1/2)$-geodesics have prescribed $g_u$-geodesic curvature equal to $1/2$. Thus, each unit speed $(g_u, 1/2)$-geodesic connected segment in the region $C_{1/4} = [-1/4,1/4] \times S^1$ can be tangent to a parallel $\alpha^{\pm}_s$ at most once, and cannot be asymptotic to $\alpha^{\pm}_s$. Moreover, \cite[Lemma 7.1]{burns2002anosov} establishes that the amount of time unit speed $(g_u, 1/2)$-geodesics spend in $C_\delta$ is bounded by $2t_\delta$ for a constant $t_\delta > 0$. Therefore, every $(\bar{g}_u, 1/2)$-geodesic on $M$ is contained in a region with negative magnetic curvature equal to $-3/4$, except on intervals with length bounded by $2t_\delta$ where the curvature is bounded by 1/2. Consider the \emph{magnetic Jacobi equation} along a $(\bar{g}_u, 1/2)$-geodesic $\gamma$: 
		\begin{equation}\label{jacobi}
			\ddot{y} + 
			K^{g_u, 1/2}y =0. \end{equation} Functions of the form $u(t) = \dot{y}(t)/y(t)$, where $y(t)$ is the unique solution of \eqref{jacobi} with initial condition $y(0)\neq 0$, are examples of solutions to \eqref{eqn:riccati} along $\gamma$. It is easy to check that for this choice of parameters, the assumptions of Lemma ~\ref{lem:exp_dichotomy} are satisfied for $T = 1/8$ and $H > 0$ depending on $t_\delta$, hence the magnetic flow associated to $(\bar{g}_u, 1/2)$ is Anosov. 	\end{proof}
	

	The goal is to now modify this argument in order to derive the following.

	\begin{lem} \label{lem:upgraded_burns_example}
		There exists a closed, connected, orientable surface $M$ and an exact Anosov magnetic system $(g, b)$ on $M$ such that $g$ is not Anosov.
	\end{lem}

	We briefly sketch the idea of the proof. The first step is to modify Lemma ~\ref{lem:burns_example} in order to allow a magnetic intensity $b_+$ which is a small positive bump function. In particular, this will show that any  $(g,b_+)$-geodesic must spend a small amount of time in the ``bad region,'' and we can shrink this amount of time as necessary. A symmetric argument will also show that one can use the magnetic intensity $b_- = -b_+$, and any $(g, b_-)$ geodesic will also spend the same amount of time in the ``bad region.'' By gluing together these examples and taking the magnetic intensity $b \coloneqq b_+ + b_-$ (note that these are supported on two disjoint regions), we see that any magnetic geodesic can only spend a small amount of time in the ``bad regions,'' and adjusting parameters can shrink this amount of time to be arbitrarily small. As long as we take $T$ sufficiently large depending on the distance between the ``bad regions,'' as well as their size, then we can find an $H$ where Lemma ~\ref{lem:exp_dichotomy} applies. Moreover, since these ``bad regions,'' have the same size, we see that the average of $b$ with respect to the area is zero, and hence this gives an example of an exact Anosov magnetic system whose underlying metric is not Anosov. We now present the details.

	\begin{proof}[Proof of Lemma ~\ref{lem:upgraded_burns_example}]
		We use the same notation as in the sketch of Lemma ~\ref{lem:burns_example}. The goal is to construct a magnetic intensity which is compactly supported on $N$, and whose growth is small. This will allow us to control the magnetic curvature as in the previous argument. 
		Fix $\epsilon > 0$ sufficiently small, and let $\delta_2 >\delta_1 \gg 1/4$. Let $v$ be a $C^0$ perturbation of $u$ such that $v$ coincides with $u$ on the interval $[-\delta_1, \delta_1]$, $K^{g_v}(s) \leq 1 $ for $|\delta_1|\leq|s|\leq |\delta_2|$ and  $K^{g_v}(s) = -1 - \epsilon $ for $|s|>\delta_2$.  
		
		\begin{figure}[H]
			\centering
				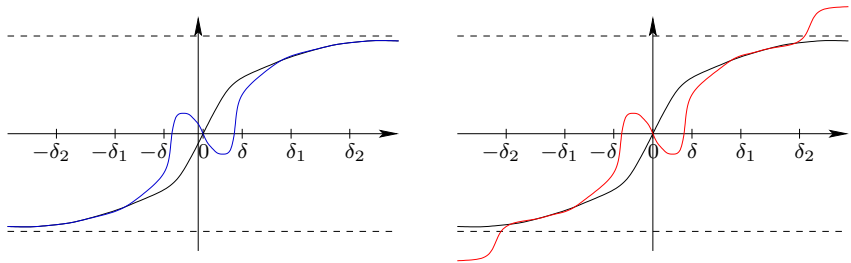
			\caption{The graphs of the functions $u$ and $v$, each sketched against the graph of $\tanh$.}
			\label{fig:magnetic_1}
		\end{figure}
		
		Let $\delta_4\gg \delta_3 \gg \delta_2$. Let $b_+$ be a bump function with the following properties:
		\begin{itemize}
			\item $b_+$ is supported in $[-\delta_4,\delta_4]$
			\item $0\leq b_+ \leq 1/2$ and $b_+\equiv 1/2$ on $[-\delta_2, \delta_2]$,
			\item $\|\mathrm{d}b_+\|_{\infty} + b_+^2 < 1/4 + \epsilon$.
		\end{itemize}

		\begin{figure}[H]
			\centering
			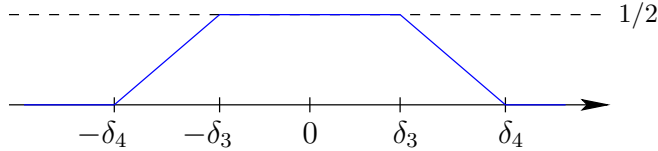
			\caption{An example of $b_+$.}
			\label{fig:magnetic_2}
		\end{figure}

		The magnetic curvature associated to $(g_v,b_+)$ satisfies the property that $K^{g,b_+}(s) = K^{g,1/2} \leq 1/2 $ for $s\in [-\delta, \delta]$, and $K^{g,b_+}(s) \leq -3/4$ otherwise. By construction, the dynamics of $(g_v,b_+)$ restricted on $C_{1/4}$ coincides with the dynamics of $(g,1/2)$. Involving the same argument as in Lemma ~\ref{lem:burns_example}, we can conclude that the time any $(g_v, b_+)$-geodesic spends in $C_\delta$ is bounded above by $2t_\delta$ for a constant $t_\delta > 0$ depending on the parameters. Let $b_- \coloneqq - b_+$. By symmetry, the time any $(g_v, b_-)$-geodesic spends in $C_\delta$ is bounded above by $2t_\delta$ for the same constant $t_\delta > 0$. For both of these examples, preform the same compactification procedure to get a closed, connected, oriented surface, and two magnetic systems $(\bar{g}_v, b_{\pm})$.
		
		Away from the neighborhoods $C_\delta$, we take two open sets and glue together two copies of $(M, \bar{g}_v)$ along these open sets to get $\bar{M} = \left[ M \cup M \right]_{g_v}$. Equip this new surface $\bar{M}$ with the magnetic system $(\bar{g}_v, b)$, where $b \coloneqq b_+ + b_-$. By construction, the time each unit speed $(g,b)$-geodesic connected segment spends in the region corresponding through the glue to $C_{\delta}$ is still bounded by $t_{\delta}$. Outside those regions, the magnetic curvature is negative and still bounded by $-3/4$, and thus by taking $T$ sufficiently large depending on the distance between the regions $C_\delta$, and taking $H$ depending on $t_\delta$ and this distance, one can conclude from Lemma ~\ref{lem:exp_dichotomy} that the corresponding flow is Anosov. Moreover, we have 	
		\[\int_{\bar{M}} b \Omega_{g} = 2\pi \left[ \int_{-\delta_4}^{\delta_4} b_+ \mathrm{d}s +\int_{-\delta_4}^{\delta_4} b_- \mathrm{d}s  \right] = 0, \]
		hence the system is exact. 
	\end{proof}

	\bibliography{MMLS}
	\bibliographystyle{plain}
\end{document}